\documentclass[12pt]{article}
\usepackage{amssymb, url}
\usepackage{amsmath,amsthm}
\usepackage{verbatim}
\usepackage{graphicx}
\usepackage{epstopdf}
\usepackage{fullpage}
\usepackage[hang,flushmargin]{footmisc}

 \newenvironment{cem}
{
    \begin{enumerate}
        \setlength{\topsep}{0pt}
        \setlength{\parskip}{0pt}
        \setlength{\partopsep}{0pt}
        \setlength{\parsep}{0pt}         
        \setlength{\itemsep}{0pt} 
}
{
    \end{enumerate} 
}

\newtheorem{theorem}{Theorem}

\newtheorem{lemma}{Lemma}
\newtheorem{conjecture}{Conjecture}

\newtheorem{proposition}{Proposition}
\newtheorem{problem}{Problem}

\newcommand{\st}{\widetilde{\sigma}}

\begin{document}

\title{On the Sum Necessary to Ensure that a Degree Sequence is Potentially $H$-Graphic}
\author{Michael J.\ Ferrara$^{1,2}$ \qquad Timothy D.\ LeSaulnier$^{3}$ \qquad Casey K.\ Moffatt$^{1}$\\
 Paul S.\ Wenger$^{4}$}
 
\maketitle
\footnotetext[1]{Department of Mathematical and Statistical Sciences, University of Colorado Denver, Denver, CO 80217 ; email addresses: {\tt michael.ferrara@ucdenver.edu}, {\tt casey.moffatt@ucdenver.edu}.}
\footnotetext[2]{Research Supported in part by Simons Foundation Collaboration Grant \#206692.}
\footnotetext[3]{Department of Mathematics, University of Illinois, Urbana, IL 61801; {\tt tlesaul2@gmail.com}.}
\footnotetext[4]{School of Mathematical Sciences, Rochester Institute of Technology, Rochester, NY 14623; {\tt pswsma@rit.edu}}

\begin{abstract}

Given a graph $H$, a graphic sequence $\pi$ is {\it potentially $H$-graphic} if there is some realization of $\pi$ that contains $H$ as a subgraph.  In 1991, Erd\H{o}s, Jacobson and Lehel posed the following question:\begin{center} Determine the minimum even integer $\sigma(H,n)$ such that every $n$-term graphic sequence with sum at least $\sigma(H,n)$ is potentially $H$-graphic.\end{center}
This problem can be viewed as a ``potential" degree sequence relaxation of the (forcible) Tur\'{a}n problems.
  
While the exact value of $\sigma(H,n)$ has been determined for a number of specific classes of graphs (including cliques, cycles, complete bigraphs and others), very little is known about the parameter for arbitrary $H$.  In this paper, we determine $\sigma(H,n)$ asymptotically for all $H$, thereby providing an Erd\H{o}s-Stone-Simonovits-type theorem for the Erd\H{o}s-Jacobson-Lehel problem.

\end{abstract}

\section{Introduction}\label{section:Intro}

A sequence of nonnegative integers $\pi =(d_1,d_2,...,d_n)$ is {\it graphic} if there is a (simple) graph $G$ of order $n$ having degree sequence $\pi$.  In this case, $G$ is said to {\it realize} or be a {\it realization of} $\pi$, and we will write $\pi=\pi(G)$.  If a sequence $\pi$ consists of the terms $d_1,\dots,d_t$ having multiplicities $\mu_1,\ldots,\mu_t$, we may write $\pi=(d_1^{~ \mu_1},\ldots,d_t^{~\mu_t})$.  Unless otherwise noted, throughout this paper all sequences are nonincreasing.  Additionally, we let $\sigma(\pi)$ denote the sum of the terms of $\pi$.

The study of graphic sequences dates to the 1950s, and includes the characterization of graphic sequences by Havel \cite{hav} and Hakimi \cite{hak} and an independent characterization by Erd\H{o}s and Gallai \cite{EG}.   Subsequent research sought to describe those graphic sequences that are realized by graphs with certain desired properties.  Such problems can be broadly classified into two types, first described as ``forcible" problems and ``potential" problems by A.R. Rao in \cite{Ra79}.  In a forcible degree sequence problem, a specified graph property must exist in every realization of the degree sequence $\pi$, while in a potential degree sequence problem, the desired property must be found in at least one realization of $\pi$. 

 Results on forcible degree sequences are often stated as traditional problems in structural or extremal graph theory, where a necessary and/or sufficient condition is given in terms of the degrees of the vertices (or equivalently the number of edges) of a given graph (e.g. Dirac's Theorem on hamiltonian graphs or the number of edges in a maximal planar graph).  Two older, but exceptionally thorough surveys on forcible and potential problems are due to Hakimi and Schmeichel \cite{HaSc78} and S.B. Rao \cite{SRa81}.

A number of degree sequence analogues to classical problems in extremal graph theory appear throughout the literature, including potentially graphic sequence variants of Hadwiger's Conjecture \cite{DM,RoSo10}, the Sauer-Spencer graph packing theorem \cite{BFHJKW}, and the Erd\H{o}s-S\' os Conjecture \cite{LiYin09}.

Pertinent to our work here is the {\it Tur\'{a}n Problem}, one of the most well-established and central problems in extremal graph theory.

\begin{problem}[The Tur\'{a}n Problem]
Let $H$ be a graph and $n$ be a positive integer.  Determine the minimum integer $ex(H,n)$ such that every graph of order $n$ with at least $ex(H,n)+1$ edges contains $H$ as a subgraph.
\end{problem}

We refer to $ex(H,n)$ as the {\it extremal number} or {\it extremal function} of $H$.  Mantel \cite{Man07} determined $ex(K_3,n)$ in 1907 and  Tur\'{a}n \cite{Tur41} determined $ex(K_t,n)$ for all $t\ge 3$ in 1941, a result considered by many to mark the start of modern extremal graph theory.  Outside of these results, the exact value of the extremal function is known for very few graphs (cf. \cite{BK,CGPW,EFGG}).  In 1966, however, Erd\H{o}s and Simonovits \cite{ESS} extended previous work of Erd\H{o}s and Stone \cite{ES} and determined $ex(H, n)$ asymptotically for arbitrary $H$.  More precisely, this seminal theorem gives exact asymptotics for $ex(H,n)$ when $H$ is a nonbipartite graph.

\begin{theorem}[The Erd\H{o}s-Stone-Simonovits Theorem]
If $H$ is a graph with chromatic number $\chi(H)\ge 2$, then $$ex(H,n)=\left( 1-\frac{1}{\chi(H)-1}\right){n\choose 2} + o(n^2).$$
\end{theorem}

   Given a family ${\mathcal F}$ of graphs, a graphic sequence $\pi$ is {\it potentially ${\mathcal F}$-graphic} if there is a realization of $\pi$ that contains some $F\in {\mathcal F}$ as a subgraph.  If ${\mathcal F}=\{H\}$, we say that $\pi$ is potentially $H$-graphic.  The focus of this paper is the following problem posed by Erd\H{o}s, Jacobson and Lehel in 1991 \cite{EJL}.  

\begin{problem}\label{problem:potential}
 Given a graph $H$, determine $\sigma(H,n)$, the minimum even integer such that every $n$-term graphic sequence $\pi$ with $\sigma(\pi)\ge \sigma(H,n)$ is potentially $H$-graphic.
\end{problem}

  We will refer to $\sigma(H,n)$ as the {\it potential number} or {\it potential function} of $H$.  As $\sigma(\pi)$ is twice the number of edges in any realization of $\pi$, Problem~\ref{problem:potential} can be viewed as a potential degree sequence relaxation of the Tur\'{a}n problem.  
  
  In \cite{EJL}, Erd\H{o}s, Jacobson and Lehel conjectured that $\sigma(K_t,n)=(t-2)(2n-t+1)+2$.  The cases $t=3,4$ and 5 were proved separately (see respectively \cite{EJL}, \cite{GJL} and \cite{LiSong2}, and \cite{LiXia}), and Li, Song and Luo \cite{LiSongLuo} proved the conjecture true for $t \geq 6$ and $n \geq {{t} \choose{2}}+3$.  In addition to these results for complete graphs, the value of $\sigma(H,n)$ has been determined exactly for a number of other specific graph families, including complete bipartite graphs \cite{ChenLiYin, LiYin2}, disjoint unions of cliques \cite{F}, and the class of graphs with independence number two \cite{FS} (for a number of additional examples, we refer the reader to the references of \cite{FS}).  Despite this, relatively little is known in general about the potential function for arbitrary $H$.  In this paper, we determine $\sigma(H,n)$ asymptotically for all $H$, thereby giving a potentially $H$-graphic sequence analogue to the Erd\H{o}s-Stone-Simonovits Theorem.

\section{Constructions and Statement of Main Result}

We assume that $H$ is an arbitrary graph of order $k$ with at least one nontrivial component and furthermore that $n$ is sufficiently large relative to $k$.  We let $\Delta(F)$ denote the maximum degree of a graph $F$, and let $F<H$ denote that $F$ is an induced subgraph of $H$.  For each $i\in\{\alpha(H)+1,\ldots,k\}$ let $$\nabla_i(H)=\min\{\Delta (F):F < H, |V(F)|=i\},$$ and consider the sequence  $$\widetilde{\pi}_i(H,n) = ((n-1)^{k-i}, (k-i+\nabla_i(H)-1)^{n-k+i}).$$
If this sequence is not graphic, that is if $n-k+i$ and $\nabla_i(H)-1$ are both odd, we reduce the smallest term by one. To see that this yields a graphic sequence, we make two observations.  First, $(\nabla_i(H)-1)$-regular graphs of order $n-k+i \ge \nabla_i(H)$ exist whenever $\nabla_i(H)-1$ and $n-k+i$ are not both odd.  If $n-k+i$ and $\nabla_i(H)-1$ are both odd, it is not difficult to show that the sequence $((\nabla_i(H)-1)^{n-k+i-1},\nabla_i(H)-2)$ is graphic.  

We first show that the sequence $\widetilde{\pi}_i(H,n)$ is not potentially $H$-graphic for each $i\in\{\alpha(H)+1,\ldots,k\}$, thus establishing a lower bound on $\sigma(H,n)$.  

\begin{proposition}\label{theorem:LowerBound}
If $H$ is a graph of order $k$ and $n$ is a positive integer, then $\sigma(H,n) \ge \sigma(\widetilde{\pi}_i(H,n)) + 2$ for all $i\in\{\alpha(H)+1,\ldots,k\}$.
\end{proposition}

\begin{proof}
Every realization $G$ of $\widetilde{\pi}_i(H,n)$ is a complete graph on $k-i$ vertices joined to an $(n-k+i)$-vertex graph $G_i$ with maximum degree $\nabla_i-1$.
Any $k$-vertex subgraph of $G$ contains at least $i$ vertices in $G_i$.
Thus $H$ is not a subgraph of $G$ since every $i$-vertex induced subgraph of $H$ has maximum degree at least $\nabla_i$.
\end{proof}

The focus of this paper is the asymptotic behavior of the potential function.  As such, let $$\widetilde{\sigma}_i(H) = 2(k-i)+\nabla_i(H)-1,$$ which is the leading coefficient of $\sigma(\widetilde{\pi}_i(H,n))$.  In \cite{FS}, the first author and J.\ Schmitt conjectured the following.

\begin{conjecture}\label{conjecture:AsymptoticOld}
Let $H$ be a graph, and let $\epsilon >0$.  There exists an $n_0=n_0(\epsilon, H)$ such that for any $n>n_0$,
$$\sigma(H,n) \leq \max_{H'\subseteq H}(\widetilde{\sigma}_{\alpha(H')+1}(H')+\epsilon)n.$$
\end{conjecture}

The condition that one must examine subgraphs of $H$ is necessary.
As an example, for $t\ge 3$ let $H$ be obtained by subdividing one edge of $K_{1,t}$.
Since $K_{1,t}$ is a subgraph of $H$, any sequence that is potentially $H$-graphic is necessarily potentially $K_{1,t}$-graphic.
However, both graphs have independence number $t$ and  $\widetilde{\sigma}_{t+1}(H)<\widetilde{\sigma}_{t+1}(K_{1,t})$.

We show that in fact one needs only examine somewhat large induced subgraphs of $H$.
The following, which determines the asymptotics of the potential function precisely for arbitrary $H$, is the main result of this paper.  

\begin{theorem}\label{theorem:potential_asymptotics}
Let $H$ be a graph of order $k$ and let $n$ be a positive integer.  If $\widetilde{\sigma}(H)$ is the maximum of $\widetilde{\sigma}_i(H)$ for $i\in\{\alpha(H)+1,\ldots,k\}$, then $$\sigma(H,n) = \widetilde{\sigma}(H)n + o(n).$$
\end{theorem}

As was pointed out in \cite{FS}, Conjecture \ref{conjecture:AsymptoticOld} is correct for all graphs $H$ for which $\sigma(H,n)$ is known.  Consequently, it is feasible that Theorem \ref{theorem:potential_asymptotics} is actually an affirmation of Conjecture \ref{conjecture:AsymptoticOld}. That is, for all $H$ is it possible that $$\widetilde{\sigma}(H) = \max_{H'\subseteq H}(\widetilde{\sigma}_{\alpha(H')+1}(H')),$$ however we are unable to either verify or disprove this at this time.  

The proof of Theorem \ref{theorem:potential_asymptotics} is an immediate consequence of Proposition \ref{theorem:LowerBound} and the following result.  

\begin{theorem}\label{theorem:PotErdosStone}
Let $H$ be a graph, and let $\omega=\omega(n)$ be an increasing function that tends to infinity with $n$.  There exists an $N=N(\omega, H)$ such that for any $n\ge N$,
\[\sigma(H,n) \le \widetilde{\sigma}(H)n+\omega(n).\]
\end{theorem}
   
The proof of Theorem \ref{theorem:PotErdosStone} relies on repeated use of the following theorem, which may be of independent interest.  Let $H$ be a graph, and let $(h_1,\dots,h_k)$ be the degree sequence of $H$.  A graphic sequence $\pi=(d_1,\dots,d_n)$ is {\it degree sufficient} for $H$ if $d_i\ge h_i$ for all $i\in\{1,\ldots,k\}$.

\begin{theorem}[The Bounded Maximum Degree Theorem]\label{theorem:BMDT}
Let $H$ be a graph of order $k$.
There exists a function $f=f(\alpha(H),k)$ such that, for sufficiently large $n$, a nonincreasing graphic sequence $\pi=(d_1,\dots,d_n)$ is potentially $H$-graphic provided that
\begin{cem}
\item $\pi$ is degree sufficient for $H$,
\item $d_n\ge k-\alpha(H)$, and
\item $d_1<n-f(\alpha(H),k)$.
\end{cem}
\end{theorem}

%

In Section \ref{section:Lemmas} we present several technical lemmas used in the proofs of Theorems \ref{theorem:PotErdosStone} and \ref{theorem:BMDT}.  In Section \ref{section:BoundMaxDeg} we prove the Bounded Maximum Degree Theorem, with the proof of Theorems \ref{theorem:LowerBound} and \ref{theorem:PotErdosStone} following in Section \ref{section:MainProof}.

\section{Technical Lemmas}\label{section:Lemmas}

We will need the following results from \cite{YL05} and \cite{YinSplit}.

\begin{theorem}[Yin and Li]\label{theorem:YL}
Let $\pi=(d_1,\dots,d_n)$ be a nonincreasing graphic sequence and let $k$ be a positive integer.  If $d_k\ge k-1$ and $d_i\ge 2(k-1)-i$ for all $i\in\{1,\ldots,k-1\}$, then $\pi$ is potentially $K_k$-graphic.
\end{theorem}

We let $G\vee H$ denote the standard join of $G$ and $H$ and let $\overline{G}$ denote the complement of $G$.

\begin{lemma}[Yin]\label{lemma:SplitGraph}
If $\pi$ is a potentially $K_r\vee\overline{K_{s}}$-graphic sequence, then there is a realization of $\pi$ in which the vertices in the copy of $K_r$ are the $r$ vertices of highest degree and the vertices in the copy of $\overline{K_s}$ are the next $s$ vertices of highest degree.
\end{lemma}

We now present a classical result of Kleitman and Wang \cite{KW} that generalizes the Havel-Hakimi algorithm \cite{hak,hav}.

\begin{theorem}[Kleitman and Wang]\label{theorem:KleitmanWang}
Let $\pi=(d_1,\dots,d_n)$ be a nonincreasing sequence of nonnegative integers.
If $\pi_{(i)}$ is the sequence defined by
$$\pi_{(i)}=\begin{cases}(d_1-1,\dots, d_{d_i}-1,d_{d_i+1},\dots,d_{i-1},d_{i+1},\dots,d_n) & \text{if}~d_i<i\\ (d_1-1,\dots,d_{i-1}-1,d_{i+1}-1,\dots,d_{d_i+1}-1,d_{d_i+2},\dots,d_n) & \text{if}~d_i\ge i, \end{cases}$$
then $\pi$ is graphic if and only if $\pi_{(i)}$ is graphic.
\end{theorem}

The process of removing a term from a graphic sequence as described in the Kleitman-Wang algorithm is referred to as {\it laying off} the term $d_i$ from $\pi$, and the sequence $\pi_i$ obtained is frequently called a \textit{residual} sequence.

The next lemma, the proof of which is essentially identical to that of the necessity of the Havel-Hakimi and Kleitman-Wang algorithms, allows us to use the structure of a realization of $\pi$ to choose the term that we wish to lay off in the Kleitman-Wang algorithm.

\begin{lemma}\label{lemma:delete_vertex}
Let $G$ be a graph with degree sequence $\pi$ and let $v\in V(G)$.  If $\pi(G-v)$ is potentially $H$-graphic, then $\pi$ is potentially $H$-graphic.  In particular, if $\pi_i$ is any residual sequence obtained from $\pi$ via the Kleitman-Wang algorithm and $\pi_i$ is potentially $H$-graphic, then $\pi$ is potentially $H$-graphic.  
\end{lemma} 

\begin{proof}
Let $\pi=(d_1,\dots,d_n)$ and let $V(G)=\{v_1,\dots,v_n\}$ such that $d(v_i)=d_i$.
Assume that $\pi_i$ is potentially $H$-graphic and let $G'$ be a realization of $\pi(G-v_i)$ that contains $H$ as a subgraph.
If $d_i<i$, obtain a realization of $\pi$ containing $H$ by joining $v_i$ to vertices with degrees $d_1-1,\ldots,d_{d_{i}}-1$ in $G'$.
If $d_i\ge i$, obtain a realization of $\pi$ containing $H$ by joining $v_i$ to vertices with degrees $d_1-1,\ldots,d_{i-1}-1,d_{i+1}-1,\ldots,d_{d_{i}+1}-1$ in $G'$.
\end{proof}

In the same spirit, we also give the following useful lemma.  As the proof is straightforward, we omit it here.  For a graph $G$ and an integer $i$ with $i\in\{0,\ldots,|V(G)|\}$, let ${\mathcal D}^{(i)}(G)$ denote the family of graphs obtained by deleting exactly $i$ vertices from $G$ or, in other words, the family of induced subgraphs of $G$ with order $|V(G)|-i$.  

\begin{lemma}\label{lemma:conical}
If $\pi=(n-1,d_2,\dots,d_n)$ is a nonincreasing graphic sequence, then $\pi$ is potentially $H$-graphic if and only if $\pi_1=(d_2-1,\dots,d_n-1)$ is potentially $H'$-graphic for some $H'$ in ${\mathcal D^{(1)}}$.  
\end{lemma}

Iteratively applying the Kleitman-Wang algorithm yields the following, which generalizes related lemmas from \cite{CFGS, F}.

\begin{lemma}\label{lemma:KleitmanWangIt}
Let $m$ and $k$ be positive integers and let $\omega=\omega(n)$ be an increasing function that tends to infinity with $n$.
There exists $N = N(m,k,\omega)$ such that for all $n\ge N$, if $\pi$ is an $n$-term graphic sequence such that $\sigma(\pi)\ge mn+\omega(n)$, then
\begin{cem}
\item $\pi$ is potentially $K_k$-graphic, or
\item iteratively applying the Kleitman-Wang algorithm by laying off the minimum term in each successive sequence will eventually yield a graphic sequence $\pi'$ with $n'$ terms satisfying the following properties: (a) $\sigma(\pi')\ge mn'+\omega(n')$,  (b) the minimum term in $\pi'$ is greater than $m/2$, and (c) $n'\ge \frac{\omega(n)}{2(k-2)-m}$.
\end{cem}

\end{lemma}

\begin{proof}
First observe that if $k\le 2$, then the result is trivial, so we assume that $k\ge 3$.
Let $\pi = \pi_0$.
If the minimum term of $\pi_i$ is greater than $m/2$, let $\pi'=\pi_i$.
Otherwise, obtain the $n-i-1$-term graphic sequence $\pi_{i+1}$ by applying the Kleitman-Wang algorithm to lay off a minimum term from $\pi_{i}$ and then sorting the resulting sequence in nonincreasing order.
Note that always $\sigma(\pi_{i+1})\ge \sigma(\pi_{i})-m$, and by induction we have $\sigma(\pi_{i+1})\ge m(n-i-1)+\omega(n)$.

By Lemma~\ref{lemma:delete_vertex}, if $\pi_i$ is potentially $K_k$-graphic, then $\pi_0$ is potentially $K_k$-graphic.
By the results of \cite{EJL}, \cite{GJL} and \cite{LiSong2}, \cite{LiXia}, and \cite{LiSongLuo}, we know that $\sigma(K_k,n) = 2(k-2)n-(k-1)(k-2)+2$ for $n\ge {k\choose 2}+3$.
Thus, if $\sigma(\pi_i)\ge 2(k-2)(n-i)+2$ for some $i\ge 0$, then $\pi$ is potentially $K_k$-graphic.
Considering the case when $i=0$, we may assume that $m<2(k-2)$.
If $i\ge n-\frac{\omega(n)}{2(k-2)-m}$, then
\begin{align*}
\sigma(\pi_i)-2(k-2)(n-i)&\ge (m-2(k-2))(n-i)+ \omega(n)\\
&\ge (m-2(k-2))\left(\frac{\omega(n)}{2(k-2)-m}\right)+\omega(n)\\
&\ge 0.
\end{align*}
Therefore, if $i\ge n-\frac{\omega(n)}{2(k-2)-m}$ and $\frac{\omega(n)}{2(k-2)-m}\ge \binom{k}{2}+3$, then $\pi_i$ is potentially $K_k$-graphic and consequently $\pi$ is potentially $K_k$-graphic.

It now follows that if $\pi$ is not potentially $K_k$-graphic, then the process of applying the Kleitman-Wang algorithm must output $\pi' = \pi_i$ for some $i<n-\frac{\omega(n)}{2(k-2)-m}$.
Since $\omega$ is increasing and $n'\le n$, it follows that $\sigma(\pi')\ge mn'+\omega(n')$.
\end{proof}

As demonstrated in Lemma \ref{lemma:KleitmanWangIt}, repeatedly laying off minimum terms from a sequence via the Kleitman-Wang algorithm can allow us to obtain a residual sequence that is denser than $\pi$ and also has a large minimum degree, which may facilitate the construction of a realization that contains $H$.  At times, we may instead wish to delete vertices from specific realizations of $\pi$.

Finally, for completeness, we state the classic characterization of graphic sequences due to Erd\H{o}s and Gallai \cite{EG}, which will be useful in the proof of Theorem \ref{theorem:PotErdosStone}.

\begin{theorem}[The Erd\H{o}s-Gallai Graphicality Criteria]\label{theorem:EG}
If $\pi=(d_1,\dots,d_n)$ is a nonincreasing sequence of nonnegative integers, then $\pi$ is graphic if and only if $\pi$ has even sum and $$\sum_{j=1}^{t}d_j\le t(t-1) + \sum_{j=t+1}^n\min\{d_j, t\}$$ for all $t\in\{1,\ldots,n-1\}$.  
\end{theorem}

\section{Proof of the Bounded Maximum Degree Theorem}\label{section:BoundMaxDeg}

In this section, we prove the Bounded Maximum Degree Theorem (Theorem~\ref{theorem:BMDT}).   Both here and in the proof of Theorem \ref{theorem:PotErdosStone}, we demonstrate the existence of a realization of $\pi$ that either contains $H$ or some supergraph of $H$, for instance $K_k$ or $K_{k-\alpha(H)}\vee \overline{K}_{\alpha(H)}$.  

Let $e_1=u_1v_1$ and $e_2=u_2v_2$ be edges in a graph $G$ such that $u_1u_2$ and $v_1v_2$ are not in $E(G)$.  Removing $e_1$ and $e_2$ from $G$ and replacing them with the edges $u_1u_2$ and $v_1v_2$ results in a graph with the same degree sequence as $G$.  This operation is called a {\it 2-switch}, but throughout the literature on degree sequences has also been referred to as a \textit{swap}, \textit{rewiring} or \textit{infusion}.  In \cite{P}, Petersen showed that, given realizations $G_1$ and $G_2$ of a graphic sequence $\pi$, $G_1$ can be obtained from $G_2$ via a sequence of 2-switches.  Both here and in Section \ref{section:MainProof}, we will utilize both 2-switches and more general edge-exchange operations.
Let $C$ be a list of vertices $C=v_0,v_1,\ldots,v_t$ with $v_t=v_0$ in a graph $G$ such that $v_{i-1}v_i \in E(G)$ if and only if $v_iv_{i+1}\notin E(G)$.
Observe that removing the edges in $E(C)\cap E(G)$ from $G$ and adding the edges in $E(C)\cap E(\overline{G})$ to $G$ preserves the degree sequence of $G$. 
\begin{proof}

Let $V(H)=\{u_1,\ldots,u_k\}$, indexed such that $d_H(u_i)\geq d_H(u_j)$ when $i\leq j$.
Let us assume that $\pi$ satisfies the hypothesis of the theorem, but is not potentially $H$-graphic.  In a realization $G$ of $\pi$, let $S=\{v_1,\ldots,v_k\}$ be a set of vertices such that $d_G(v_i)=d_i$ and let $H_S$ be the graph with vertex set $S$ in which two vertices $v_i,v_j$ are adjacent if and only $u_iu_j\in E(H)$.  If all edges of $H_S$ are edges of $G$, then $H_S$ is a subgraph of $G$ isomorphic to $H$ and $\pi$ is potentially $H$-graphic.  Hence, let us assume that $G$ is a realization which maximizes $|E(G)\cap E(H_S)|$, but that the edge $v_iv_j\in E(H_S)$ while $v_iv_j\not\in E(G)$.  Now, because $d_G(v_i)\geq d_H(u_i)=d_{H_S}(v_i)$ and $d_G(v_j)\geq d_H(u_j)=d_{H_S}(v_j)$, there must exist (not necessarily distinct) vertices $a_i$ and $a_j$ such that $v_ia_i,v_ja_j\in E(G)$, while $v_ia_i,v_ja_j\not\in E(H_S)$.  

We begin by showing that many vertices of $V(G)-S$ have neighbors in $V(G)-S$.  Specifically, we claim that there are at most $$g(\alpha(H),k) = {k\choose k-\alpha(H)}\left[2{k-\alpha(H)\choose 2}+\alpha(H)-1\right] $$ vertices $w$ in $V(G)-S$ such that $N(w)\subseteq S$.

Indeed, let us assume that there are at least $g(\alpha(H),k)+1$ vertices $w$ in $V(G)-S$ such that $N(w)\subseteq S$.  By the hypothesis of the theorem, each vertex $w\in V(G)$ satisfies $d_G(w)\geq k-\alpha(H)$.  For each vertex $w$ such that $N(w)\subseteq S$, let $S_w$ be a set of $k-\alpha(H)$ vertices such that $S_w\subseteq N(w)$.  Let $\ell=2{k-\alpha(H)\choose 2}+\alpha(H)$.  By the pigeonhole principle, there exists an independent set $\{w_1,\ldots,w_\ell\}$ of vertices and a $k-\alpha(H)$ element subset $\widehat{S}$ of $S$ such that $S_{w_i}=\widehat{S}$ for each $1\le i\le\ell$.    

Let $\mathcal{P}$ be a family of ${k-\alpha(H)\choose 2}$ disjoint pairs of vertices from $\{w_1,\ldots,w_{\ell-\alpha(H)}\}$, and to each pair $\{v_r,v_s\}$ of vertices in $\widehat{S}$ associate a distinct pair $P(r,s)\in\mathcal{P}$.  If $v_rv_s$ is not an edge in $G$ and we suppose $P(r,s)=\{w_{r'}.w_{s'}\}$, then we perform the 2-switch that replaces the edges $v_rw_{r'}$ and $v_sw_{s'}$ with the non-edges $v_rv_s$ and $w_{r'}w_{s'}$.  In this way arrive at a realization of $\pi$ in which the graph induced by $\widehat{S}$ is complete, while maintaining the property that the vertices of $\widehat{S}$ are joined to $\{w_{\ell-\alpha(H)+1}.\ldots,w_\ell\}$.  This produces a realization of $\pi$ that contains $K_{k-\alpha(H)}\vee \overline{K_{\alpha(H)}}$ as a subgraph, contradicting the assumption that $\pi$ is not potentially $H$-graphic.  Consequently, we may assume that there are at most $g(\alpha(H),k)$ vertices $w$ in $V(G)-S$ such that $N(w)\subseteq S$.

We will now use the fact that many vertices of $V(G)-S$ have neighbors in $V(G)-S$ to exhibit an edge-exchange that inserts the edge $v_iv_j$ into $G$ at the expense of the edges $v_ia_i$ and $v_ja_j$ while preserving each edge in $H_S$.  Let $$f(\alpha(H),k)=[g(\alpha(H),k)+4k^2]+k+1,$$ and let $d_1\leq n-1-f(\alpha(H),k)$.  If we let $$X_i=\{v\in (V(G)-S)-N_{G-S}(a_i)~:~d_{G-S}(v)>0\}$$ and $$X_j=\{v\in (V(G)-S)-N_{G-S}(a_j)~:~d_{G-S}(v)>0\},$$ then, as at most $g(\alpha,k)$ vertices have their neighborhoods entirely contained in $S$, it follows that $$|X_i|, |X_j|\ge 4k^2.$$
Let $Y_i = N_{G-S}(X_i)$ and $Y_j=N_{G-S}(X_j)$.

By assumption, $\pi$ is not potentially $H$-graphic, and is therefore not potentially $K_k$-graphic.  As such, Theorem~\ref{theorem:YL} implies that each vertex in $G-S$ has degree at most $2k-4<2k$.
Let $y_i$ be a vertex in $Y_i$, and let $x_i$ be a neighbor of $y_i$ in $X_i$.
There are at least $4k^2-1$ vertices in $X_j$ that are not $x_i$, and at least $4k^2-1-(2k-5)$ of these vertices have a neighbor in $Y_j$ that is not $y_i$.
Since vertices in $V(G)-S$ have degree at most $2k-4$, there are more than $(4k^2-2k+4)/(2k) = 2k-1+2/k$ vertices in $Y_j$ that are not $y_i$ and furthermore have a neighbor in $X_j$ that is not $x_i$.
Since $d(y_i)<2k$, there exists a vertex $y_j$, distinct from $y_i$ such that $y_iy_j\notin E(G)$ and $y_j$ has a neighbor $x_j\in X_j$ that is distinct from $x_i$.
In this case, exchanging the edges $v_ia_i,v_ja_j, x_iy_i$ and $x_jy_j$ for the nonedges $v_iv_j, a_ix_i,a_jx_j$ and $y_iy_j$ yields a realization of $\pi$ that contradicts the maximality of $G$.
\end{proof}

\section{Proof of Theorem \ref{theorem:PotErdosStone}}\label{section:MainProof}

The proof of Theorem \ref{theorem:PotErdosStone} proceeds in two stages, which
we briefly outline here to give a clearer picture of the structure and philosophy of the proof.  Let $n$ be sufficiently large, and let $\pi$ be an $n$-term graphic sequence such that $\sigma(\pi)\ge \st(H)n+\omega(n)$.  Most simply stated, we prove that $\pi$ has a realization in which a set of $\ell$ vertices is completely joined to a large subgraph whose degree sequence is potentially $\mathcal D^{(\ell)}(H)$-graphic, for some $\ell\in\{0,\ldots,k-\alpha(H)\}$.

The first stage of the proof generates a residual sequence $\pi_{\ell}$ through $\ell$ alternating applications of Lemma~\ref{lemma:KleitmanWangIt} and (the contrapositive of) Theorem~\ref{theorem:BMDT}.
The $i$th application of Lemma~\ref{lemma:KleitmanWangIt} raises the average of the terms in the sequence, making the sequence denser and therefore more likely to be potentially $\mathcal D^{(i)}(H)$-graphic.
We call this sequence $\pi_i$.

If $\pi_i$ is not potentially $\mathcal D^{(i)}(H)$-graphic, then the contrapositive of Theorem~\ref{theorem:BMDT} implies that the leading term of $\pi_i$ is very large.
Hence, in any realization $G$ of $\pi_i$ there is a vertex $v$ that is adjacent to nearly all other vertices in $G$.
While the leading term of $\pi_i$ prevents the direct application of Theorem~\ref{theorem:BMDT} to build an element of $\mathcal D^{(i)}(H)$, in principle the corresponding vertex of high degree should be beneficial to the construction of an element of $\mathcal D^{(i)}(H)$.
We then lay off the leading term of $\pi_i$, and given its utility, consider it to be ``reserved" for the conclusion of the proof.

After appropriate conditions are met, Stage 1 terminates with $\pi_{\ell}$.  Stage 2 of the proof then uses edge exchanges to prove that $\pi_{\ell}$ is potentially $\mathcal D^{(\ell)}(H)$-graphic.  Since Stage 1 was iterated $\ell$ times, we have laid off and reserved $\ell$ large terms.  Reuniting these $\ell$ terms with a realization of $\pi_\ell$ that contains a subgraph from $D^{(\ell)}(H)$ allows us to create a realization of $\pi$ that contains $H$ as a subgraph, completing the proof.  

We are now ready to give the proof of Theorem \ref{theorem:PotErdosStone} in full detail.  

\begin{proof}
To begin, let $\pi$ be an $n$-term graphic sequence with $n>N(\omega, H)$ such that $\sigma(\pi)\ge \st(H)n+\omega(n)$.  It suffices to show that $\pi$ is potentially $H$-graphic.
If $\pi$ is potentially $K_k$-graphic, then the result is trivial.
Therefore, by Theorem~\ref{theorem:YL}, we may assume that $d_{k+1}\le 2k-4$.
Furthermore, throughout the proof we may assume that when applying Lemma~\ref{lemma:KleitmanWangIt} to $\pi$ or a residual sequence obtained from $\pi$, the resulting sequence satisfies Conclusion (2) of that lemma.

To initialize Stage 1, apply Lemma~\ref{lemma:KleitmanWangIt} to $\pi$ to obtain an $n_0$-term graphic sequence $\pi_0=(d_1^{(0)},\ldots,d_{n_0}^{(0)})$ with minimum entry at least $\st(H)/2$ and $\sigma(\pi_0)\ge \st(H)n_0+\omega(n_0)$.  
Starting with this $\pi_0$, we iteratively construct sequences $\pi_i=(d_1^{(i)},\dots,d_{n_i}^{(i)})$ for $i\ge 1$.
Each $\pi_i$, which we assume to be nonincreasing, will satisfy the following conditions: 

\begin{enumerate}
\item[(a)]  $d_{n_i}^{(i)}\ge k-i-\alpha(H)$, and
\item[(b)] if $\pi_i$ is potentially ${\mathcal D^{(i)}(H)}$-graphic, then $\pi$ is potentially $H$-graphic.  
\end{enumerate}

Note that Condition (a) ensures that $\pi_i$ satisfies Condition 2 of Theorem~\ref{theorem:BMDT} for any $(k-i)$-vertex graph $H_i$ with $\alpha(H_i)\ge\alpha(H)$.

If $i=k-\alpha(H)$ or if $d_1^{(i)}<n_i-\binom{k}{\lceil k/2\rceil}(k^2-1)$, then terminate Stage $1$ and proceed to Stage $2$ with $\ell=i$.
Otherwise, $d_1^{(i)}\ge n_i-\binom{k}{\lceil k/2\rceil}(k^2-1)$, and we obtain $\pi_{i+1}$ via the following steps;

\begin{enumerate}

\item Let $\widehat G_i$ be a realization of $\pi_i$ such that a vertex of degree $d_1^{(i)}$, which we call $v_1^{(i)}$, is adjacent to the next $d_1^{(i)}$ vertices of highest degree; such a realization exists following the proof of the sufficiency of Theorem~\ref{theorem:KleitmanWang}.
From $\widehat G_i$, obtain the graph $\widehat G_i'$ by deleting the nonneighbors of $v_1^{(i)}$ and let $\pi_i'=\pi(\widehat G_i')$.
If we let $n_i'$ denote the number of terms in $\pi_i'$, then the largest term in $\pi_i'$ is $n_i'-1$.  

\item Apply Lemma~\ref{lemma:KleitmanWangIt} to $\pi_i'$, and let $\pi_i''$ be the graphic sequence from Conclusion (2) of that lemma.
The minimum term of $\pi_i''$ is at least $k-i-\alpha(H)$.
Let $n_i''$ be the number of terms in $\pi_i''$; the largest term in $\pi_i''$ will be $n_i''-1$.  

\item Obtain $\pi_{i+1}$ by laying off the largest term in $\pi_i''$. 

\end{enumerate}

Note that even though Lemma \ref{lemma:KleitmanWangIt} may require laying off a large fraction of the terms in $\pi_i'$, we still have that $n_{i+1}\to\infty$ provided $n_i\to\infty$.  Thus we may assume that $n_{i+1}$ is sufficiently large.

We claim that $\sigma(\pi_i)\ge (\st(H)-2i)n_i+\omega(n_i)/2^i$ and the minimum term in $\pi_i$ is at least $\st(H)/2-i\ge k-i-\alpha(H)$ for all $i\in\{0\ldots,\ell\}$.  We use induction on $i$.
First observe that $\sigma(\pi_0)\ge \st(H)n_0+\omega(n_0)$ and (as we assume that Conclusion (2) results from all applications of Lemma~\ref{lemma:KleitmanWangIt}) the minimum term of $\pi_0$ is at least $\st(H)/2\ge k-\alpha(H)$.
Let
$$M=2\binom{k}{\lceil k/2\rceil}(k^2-1)(2k-4).$$
Creating $\pi_i'$ entails deleting at most $\binom{k}{\lceil k/2\rceil}(k^2-1)$ vertices each with degree at most $2k-4$, and consequently, by induction,
$$\sigma(\pi_i') \ge \sigma(\pi_i)-M\ge (\st(H)-2i)n_i+\omega(n_i)/2^i-M.$$
By assumption, $n_i$ is sufficiently large, so we have $\omega(n_i)/2^i\ge 2M$, and therefore $\sigma(\pi_i') \ge (\st(H)-2i)n_i+\omega(n_i)/2^{i+1}$.
We next apply Lemma~\ref{lemma:KleitmanWangIt} to $\pi_1'$ to obtain a sequence $\pi_i''$ with $n_i''$ terms.
It follows that $$\sigma(\pi_i'')\ge (\st(H)-2i)n_i''+\omega(n_i'')/2^{i+1}$$ and we also have that the minimum term in $\pi_i''$ is at least $k-i-\alpha(H)$.  
The maximum term in $\pi_i''$ is $n_i''-1$, and we lay off this term to obtain $\pi_{i+1}$.  This yields
\begin{align*}
\sigma(\pi_{i+1})&\ge (\st(H)-2i)n_i''+\omega(n_i'')/2^{i+1}-2(n_i''-1)\\
&\ge (\st(H)-2(i+1))n_{i+1}+\omega(n_{i+1})/2^{i+1}.
\end{align*}
Also, as the minimum term in $\pi_i''$ is at least $k-i-\alpha(H)$ and decreases by $1$ when the maximum term is laid off, we have that $\pi_{i+1}$ satisfies Condition (a).  

Next we show that Condition (b) holds, also via induction on $i$.
If we let ${\mathcal D}^{(0)}(H)=\{H\}$, then the claim holds for $i=0$.
For some $i<\ell$, assume that if $\pi_i$ is potentially $\mathcal D^{(i)}(H)$-graphic, then $\pi$ is potentially $H$-graphic.
By Lemma \ref{lemma:conical}, $\pi_{i+1}$ is potentially ${\mathcal D}^{(i+1)}(H)$-graphic if and only if $\pi_i''$ is potentially ${\mathcal D}^{(i)}(H)$-graphic.
By Lemma \ref{lemma:delete_vertex}, if $\pi_i''$ is potentially ${\mathcal D}^{(i)}(H)$-graphic, then $\pi_i$ is potentially ${\mathcal D}^{(i)}(H)$-graphic.
Therefore, by induction, if $\pi_{i+1}$ is potentially $\mathcal D^{(i+1)}(H)$-graphic, then $\pi$ is potentially $H$-graphic, so $\pi_{i+1}$ satisfies Condition (b).   

We now begin Stage 2 of the proof: it remains to show that $\pi_{\ell}$ is potentially $\mathcal D^{(\ell)}(H)$-graphic.
If $\ell=k-\alpha(H)$, then $\mathcal D^{(\ell)}(H)$ contains $\overline{K_{\alpha(H)}}$, and since $\pi_\ell$ is sufficiently long the result holds trivially.
Thus we may assume that $\ell<k-\alpha(H)$ and that $d_1^{(\ell)}<n_\ell-\binom{k}{\lceil k/2\rceil}(k^2-1)$.

In $\pi_\ell$, let $t=\max\{i: d_i^{(\ell)}\ge k-\ell-1\}$.
First assume that $t\ge k-\ell-\alpha(H)$.
In this case, $\pi_\ell$ is degree sufficient for $K_{k-\ell-\alpha(H)}\vee\overline{K_{\alpha(H)}}$.
Furthermore, the minimum term in $\pi_\ell$ is at least $k-\ell-\alpha(H)$, and $d_1^{(\ell)}<n_\ell-\binom{k}{\lceil k/2\rceil}(k^2-1)<n_\ell-f(\alpha(H),k-\ell)$.
Thus Theorem \ref{theorem:BMDT} implies that $\pi_\ell$ is potentially $K_{k-\ell-\alpha(H)}\vee\overline{K_{\alpha(H)}}$-graphic.  Since $K_{k-\ell-\alpha(H)}\vee\overline{K_{\alpha(H)}}$ is a supergraph of a $(k-\ell)$-vertex subgraph of $H$, it follows that $\pi_\ell$ is potentially $\mathcal D^{(\ell)}(H)$-graphic.

Now assume that $t< k-\ell-\alpha(H)$.
Let $F_i$ denote an $i$-vertex induced subgraph of $H$ that achieves $\Delta(F_i) = \nabla_i(H)$.
We show that $\pi_\ell$ is degree sufficient for $K_{t}\vee F_{k-\ell-t}$.
First observe that since $k-\ell-t>\alpha(H)$ we have $\st(H)-2\ell\ge \st_{k-\ell-t}(H)-2\ell=2t-\nabla_{k-\ell-t}(H)-1$.
Thus, $$\sigma(\pi_\ell)\ge (2t-\nabla_{k-\ell-t}(H)-1)n_\ell+\omega(n_\ell)/2^{\ell}.$$

However, if $\pi_\ell$ is not degree sufficient for $K_{t}\vee F_{k-\ell-t}$, then since the minimum degree of $K_t\vee F_{k-\ell-t}$ is at most $t+\nabla_{k-\ell-t}(H)$, it follows that $d^{(\ell)}_{k-\ell}\le t+\nabla_{k-\ell-t}(H)-1.$  Taken together with the fact that $d_{t+1}^{(\ell)}\le k-\ell-2$, we then have that
$$\sigma(\pi_\ell)\le t(n_\ell-1)+(k-\ell-t-1)(k-\ell-2)+(n_\ell-k+\ell+1)(t+\nabla_{k-\ell-t}(H)-1).$$ As $\ell\le k-1$, this implies that  
$$\sigma(\pi_\ell) < (2t+\nabla_{k-\ell-t}(H)-1)n_\ell+k^2,$$
contradicting the above lower bound when $n_\ell$ is sufficiently large.  
Consequently, we have that $\pi_\ell$ is degree sufficient for $K_{t}\vee F_{k-\ell-t}$.
Note that $\alpha(F_{k-\ell-t})$ may be less than $\alpha(H)$.
Therefore we are not able to immediately apply Theorem \ref{theorem:BMDT} to obtain a realization of $\pi_\ell$ containing $K_{t}\vee F_{k-\ell-t}$.

However, in this case, $\pi_\ell$ is degree sufficient for $K_t\vee\overline{K_{k-\ell-t}}$.
Since $\pi_{\ell}$ satisfies Condition (a), the minimum term of $\pi_\ell$ is at least $k-\ell-\alpha(H)$, and by assumption, $k-\ell-\alpha(H)> t$.
Furthermore, $d_1^{(\ell)}<n_\ell-\binom{k}{\lceil k/2\rceil}(k^2-1)<n_{\ell}-f(k-\ell-t,k-\ell)$.
Thus by Theorem \ref{theorem:BMDT}, there is a realization $G_\ell$ of $\pi_\ell$ containing $K_t\vee\overline{K_{k-\ell-t}}$.  If $v_1,\ldots,v_{k-\ell}$ are the $k-\ell$ vertices of highest degree in $G_\ell$, then by Lemma~\ref{lemma:SplitGraph} we may assume that $\{v_1,\ldots,v_t\}$ is a clique that is completely joined to $\{v_{t+1},\ldots,v_{k-\ell}\}$.
Delete $v_1,\ldots,v_t$ from $G_\ell$ to obtain $G_\ell'$, and let $\pi_\ell'=\pi(G_\ell')$, with the order of $\pi_\ell'$ coming from the ordering of $\pi_\ell$.
Thus the first $k-\ell-t$ terms in $\pi_\ell'$ correspond to the vertices $\{v_{t+1},\ldots,v_{k-\ell}\}$ in $G_\ell$.
Since the minimum degree of the vertices in $G_\ell$ is at least $k-\ell-\alpha(H)$ and $t<k-\ell-\alpha(H)$, the minimum term in $\pi_\ell'$ is at least $1$.
It remains to show that there is a realization of $\pi_\ell'$ that contains a copy of $F_{k-\ell-t}$ on the vertices $\{v_{t+1},\ldots,v_{k-\ell}\}$.

To construct such a realization, place a copy of $F_{k-\ell-t}$ on the vertices $v_{t+1},\ldots,v_{k-\ell}$.
Since $n_\ell$ is sufficiently large, we may join any remaining edges incident to $\{v_{t+1},\ldots,v_{k-\ell}\}$ to distinct vertices among the remaining $n_{\ell}-(k-\ell)$ vertices.
It remains to show that there is a graph on the remaining $n_{\ell}-(k-\ell)$ vertices that realizes the residual sequence.
This sequence has at least $n_{\ell}-(k-\ell)-(k-\ell-t)(k-\ell-3)$ positive terms with the maximum term being at most $k-\ell-2$.
By Theorem \ref{theorem:EG}, the Erd\H os-Gallai criteria, such a sequence is graphic provided that $n_\ell$ is sufficiently large.
\end{proof}

As noted above, Theorem \ref{theorem:potential_asymptotics} follows immediately from Theorems \ref{theorem:LowerBound} and \ref{theorem:PotErdosStone}.

\section{Conclusion}

Having determined the asymptotic value of the potential function for general $H$, a natural next step is to study the structure of those $n$-term graphic sequences that are not potentially $H$-graphic, but whose sum is close to $\sigma(H,n)$.  This line of inquiry would be related to recent work of Chudnovsky and Seymour \cite{CS}, which for an arbitrary graphic sequence $\pi$ gives a partial structural characterization of those graphic sequences $\pi'$ for which no realization of $\pi'$ contains {\it any} realization of $\pi$ as an {\it induced} subgraph.  The first stage of this investigation appears in \cite{EFMW}.\\

{\bf Acknowledgement:}  The authors would like to acknowledge the detailed and extremely helpful efforts of an anonymous referee whose suggestions have greatly improved the clarity and presentation of this paper.

\end{document}